\documentclass[12pt,a4paper]{amsart}
\usepackage{amsmath}
\usepackage{mathtools}

\DeclarePairedDelimiter{\norm}{\lVert}{\rVert}

\usepackage{amsfonts}
\usepackage{amsthm}
\usepackage{autobreak}
\usepackage[english]{babel}
\usepackage{bold-extra}
\usepackage{hyperref}
\theoremstyle{definition}
\newtheorem{definition}{Definition}
\theoremstyle{theorem}
\newtheorem{prop}[definition]{Proposition}

\newtheorem{lemma}[definition]{Lemma}
\newtheorem{theorem}[definition]{Theorem}

\theoremstyle{remark}
\newtheorem{remark}[definition]{Remark}

\numberwithin{definition}{section}
\numberwithin{exercise}{section}
\theoremstyle{theorem}
\newcommand{\thistheoremname}{}
\newtheorem*{genericthm*}{\thistheoremname}
\newenvironment{namedthm*}[1]
{\renewcommand{\thistheoremname}{#1}%
	\begin{genericthm*}}
	{\end{genericthm*}}
\allowdisplaybreaks

\def\D{{\mathbb{D}}}

\def\C{{\mathbb{C}}}

\def\N{{\mathbb{N}}}

\def\T{{\mathbb{T}}}
\def\UU{{\mathcal{U}}}

\newcommand{\vp}{\varphi}

\usepackage{multicol}
\usepackage[T1]{fontenc}
\usepackage[utf8]{inputenc}
\usepackage[a4paper]{geometry}
\usepackage{amssymb,amsmath}
\usepackage{version}
\usepackage{mathtools} 
\usepackage{enumerate}
\usepackage{dsfont}
\usepackage{mathdots}
\makeatletter

\usepackage{latexsym}
\usepackage{color}

\begin{document}
	\title[Simultaneous Approximation by Finite Blaschke Products]{Simultaneous Approximation by Finite Blaschke Products and Bounded Universal Functions}
	\author{Konstantinos Maronikolakis}
	\address{Konstantinos Maronikolakis, Department of Mathematics, Bilkent University, 06800 Ankara, Turkey}
	\email{conmaron@gmail.com}
	\date{}
	\begin{abstract}
		This paper complements the work done on simultaneous approximation results in classical Banach spaces, by focusing on approximation by finite Blaschke products. We prove the existence of a finite Blaschke product that approximates a prescribed holomorphic function bounded by 1 locally uniformly on the unit disc, and simultaneously approximates a prescribed unimodular continuous function uniformly on a compact subset of the unit circle of arclength measure 0. We also prove an analogue where the continuous function is bounded by 1 and the the approximation is achieved by an appropriate dilate of the finite Blaschke product. These results are essentially combinations of classical results of Caratheodory and Fisher on approximation by finite Blaschke products. We also give analogues for singular inner functions. Finally, we apply our results to prove the existence of bounded holomorphic functions on the unit disc that exhibit a certain universal boundary behaviour.
	\end{abstract}
	\maketitle
	\section{Introduction}
	This paper is broadly concerned with bounded holomorphic functions on the unit disc $\D$. As usual, we denote this class of functions with $H^\infty$. The set $H^\infty$ is a Banach space when we endow it with the supremum norm $\Vert\cdot\Vert_{\infty}$ on $\D$, that is, $\Vert f\Vert_{\infty}:=\sup_{z\in\D}|f(z)|$ for $f\in H^\infty$. Similarly, for a compact set $K$ in $\C$, we define the supremum norm on $K$ as $\Vert\vp\Vert_{K}=\sup_{z\in K}|f(z)|$ for $\vp$ a continuous (complex-valued) function on $K$. The space $C(K)$ of continuous functions on $K$ is also a Banach space when endowed with the supremum norm.
	
	The main objects of this paper are \emph{finite Blaschke products}, which are finite products of automorphisms of $\D$, that is, functions of the form
	$$\zeta\prod_{k=0}^{n}\frac{z-w_k}{1-\overline{w_k}z}\text{ for }z\in\D$$
	where $w_0,w_1,\dots,w_n\in\D$ and $\zeta\in\T:=\partial\D$. There is a vast literature on finite Blaschke products and they appear naturally throughout the study of bounded holomorphic functions as well as other areas, including geometry and operator theory. We refer the interested reader to the survey of Garcia, Mashreghi and Ross \cite{GarciaMashreghiRoss2018}.
	
	They also play a crucial role in Approximation Theory. The principal theorem exemplifying this role is \hyperref[caratheo]{Caratheodory's Theorem} which states that the set of finite Blaschke products is dense in the closed unit ball $\mathcal{B}:=\{f\in H^\infty:\norm{f}_\infty\leq1\}$ of $H^\infty$ with respect to the locally uniform topology \cite{Caratheodory1954}. A result of similar nature concerning the possible boundary approximation of finite Blaschke products is \hyperref[fisher_cont]{Fisher's Theorem} which states that the set of finite Blaschke products is dense in the set $\{\vp\in C(K):|\vp|\equiv1\text{ on }K\}$, where $K$ is any compact subset of $\T$ of arclength measure 0 \cite{Fisher1971}. Our main motivating question is the following.
	
	\vspace{\baselineskip}
	
	\noindent \emph{Let $L$ be a a compact subset of $\D$, $K$ be compact subset of $\T$ of arclength measure 0, $f\in\mathcal{B}$ and $\vp\in C(K)$ such that $|\vp|\equiv1$ on $K$. We know that there exists a finite Blaschke product that approximates $f$ as well as we want uniformly on $L$ (\hyperref[caratheo]{Caratheodory's Theorem}) and a possibly different finite Blaschke product that approximates $\vp$ as well as we want uniformly on $K$ (\hyperref[fisher_cont]{Fisher's Theorem}). Is it possible to find a finite Blaschke product that simultaneously achieves both approximations?}
	
	\vspace{\baselineskip}
	
	We call this phenomenon \emph{simultaneous approximation}. This question arises naturally from examining similar problems of approximation in Banach spaces of holomorphic functions. For example, we can consider, for $1\leq p<\infty$, the Hardy space $H^p$ of holomorphic functions $f$ on $\D$ such that $\lVert f\rVert_p^p:=\sup_{0<r<1}\left(\frac{1}{2\pi}\int_{0}^{2\pi}|f(re^{it})|^pdt\right)<\infty$ (we refer to \cite{Duren1970} for the theory of these spaces). It is a standard result in the theory of Hardy spaces that the set of polynomials is dense in $H^p$. Similarly, if $K$ is a compact subset of $\T$ of arclength measure 0, we have that the set of polynomials is dense in $C(K)$; one way to see this is as an application of Mergelyan's Theorem. In \cite{BeiseMuller2016}, Beise and M\"uller proved the following simultaneous approximation result: let $K$ be a compact subset of $\T$ of arclength measure 0, $f\in H^p$ and $\vp\in C(K)$. Then, for every $\varepsilon>0$, there exists a polynomial $P$ such that
	$\norm*{P-\vp}_K<\varepsilon\ \text{ and }\ \norm*{P-f}_p<\varepsilon.$
	
	These kinds of simultaneous approximation results in Banach spaces of holomorphic functions have been studied since the work of Khrushchev \cite{Khrushchev1979} and have been proved in various settings, where the required condition on the compact set $K$ is dependent on the corresponding space; see for example \cite{CostakisJungMuller2019,Muller2019,Limani2024,CharpentierEspoullierZarouf2025}. One of our main results (Theorem \ref{sim_approx_circle}) answers positively the motivational question stated above, which essentially gives us a simultaneous approximation result by finite Blaschke products. We also prove an analogue where the function $\vp$ we wish to approximate is not necessarily unimodular on $K$ but is rather an element of the closed unit ball $BC(K):=\{\vp\in C(K):\lVert\vp\rVert_K\leq1\}$, but then the approximation is accomplished by appropriate dilates of finite Blaschke products (Theorem \ref{sim_approx_disc}).
	
	Despite the inspiration by this previous work, there is a crucial difference in the proof of this fact compared to the aforementioned established results, which is that the set of polynomials, as well as the spaces in which we want to approximate are vector spaces. This allows the use of linear techniques (the Hahn-Banach Theorem usually plays a central role). In our setting, this is not the case; the set of finite Blaschke products is not a vector space and similarly the same is true for $\mathcal{B}$ and $\{\vp\in C(K):|\vp|\equiv1\text{ on }K\}$. Thus, our proof uses constructive methods. We note that constructive methods have been used before in \cite{BeneteauIvriiManolakiSeco2020} and \cite{CharpentierEspoullierZarouf2025} for the proofs of simultaneous approximation theorems.
	
	A notion that is intimately connected with simultaneous approximation is \emph{universality}. Indeed, the authors of \cite{BeiseMuller2016,CostakisJungMuller2019,Muller2019} use their results to prove the existence of so-called ``universal Taylor series'' in Hardy, Bergman and Dirichlet spaces (we refer to these papers and the references therein for the definition of universal Taylor series). Similarly, the existence of functions whose dilates can approximate anything plausible on specific compact subsets of $\T$ was proved in \cite{Maronikolakis2022} for Hardy, Bergman and Dirichlet spaces and in \cite{CharpentierEspoullierZarouf2025} for the Bloch space. A more precise and detailed account of these types of results (and universal boundary behaviour of holomorphic functions more generally) is given in Section \ref{univ}. In that section, we use our new simultaneous approximation results to prove the existence of functions in $\mathcal{B}$ (and even infinite Blaschke products) whose dilates exhibit universal properties (Theorem \ref{univ_main}).
	
	Finally, we also prove analogues of our main results where the approximation is achieved by singular inner functions instead of finite Blaschke products (see Theorem \ref{sim_approx_sing}).
	\section{Approximation by finite Blaschke products}
	\subsection{The theorems of Fisher and Caratheodory and main results}
	The approximation properties of finite Blaschke products have been studied thoroughly since the work of Caratheodory \cite{Caratheodory1954} (see also below). We present a couple results relevant to the current work here (\hyperref[fisher_cont]{Fisher's Theorem} and \hyperref[caratheo]{Caratheodory's Theorem}) and refer to \cite{GarciaMashreghiRoss2017} or \cite{GarciaMashreghiRoss2018} for the history and overview of these kinds of results.
	\begin{namedthm*}{Fisher's Theorem}[Theorem 1 (b) in \cite{Fisher1971}]\label{fisher_cont}
		Let $K$ be a compact subset of $\T$ of arclength measure $0$. Let also $\vp$ be a continuous function on $K$ with $|\vp|\equiv1$ on $K$ and $\varepsilon>0$. Then, there exists a finite Blaschke product $B$ such that
		$$\norm*{B-\vp}_K<\varepsilon.$$
	\end{namedthm*}
	We make two remarks about the previous result. The first is that the cited theorem seems at first glance weaker than the statement above. Indeed it says the following: \emph{Let $K$ be a compact subset of $\T$ of arclength measure $0$. Let also $\vp\in A(\D)$ with $\lVert\vp\rVert_\infty\leq1$, $\vp\equiv1$ on $K$ and $\varepsilon>0$. Then, there exists a finite Blaschke product $B$ such that $\norm*{B-\vp}_K<\varepsilon.$} (where we denote throughout by $A(\D)$ the algebra of the unit disc, that is $A(\D):=\{f:\overline{\D}\to\C:f\text{ is holomorphic on }\D\text{ and continuous on }\overline{\D}\}$). The two forms can be reconciled after taking into account the Rudin-Carleson Theorem (proved independently by Rudin \cite{Rudin1956} and Carleson \cite{Carleson1957}).
	\begin{namedthm*}{Rudin-Carleson Theorem}[\cite{Rudin1956}]\label{rcthm}
		Let $K$ be a compact subset of $\T$ of arclength measure $0$. Let also $\vp$ be a continuous function on $K$ and $T\subseteq\C$ a set that is homeomorphic to $\overline{\D}$. Then, there exists $f\in A(\D)$ such that $f\equiv\vp$ on $K$ and $f(\overline{\D})\subseteq T$.
	\end{namedthm*}
	If we apply this theorem for $\vp$ satisfying $|\vp|\equiv1$ on $K$ and $T=\overline{\D}$, we get the existence of a function $f\in\mathcal{B}$ such that $f\equiv\vp$ on $K$. The aforementioned equivalence of the two forms of Fisher's Theorem is then immediate.
	
	Our second remark about \hyperref[fisher_cont]{Fisher's Theorem} is that it is the lesser known result of Fisher on approximation by finite Blaschke products. Indeed, the result that is more frequently cited as Fisher's Theorem in similar contexts concerns approximation by convex combinations of finite Blaschke products \cite{Fisher1968}.
	
	The next important theorem on approximation by finite Blaschke products is a result of Caratheodory \cite{Caratheodory1954}.
	\begin{namedthm*}{Caratheodory's Theorem}[Theorem 5.1 in \cite{GarciaMashreghiRoss2017}]\label{caratheo}
		Let $L$ be a compact subset of $\D$. Let also $f\in\mathcal{B}$ and $\varepsilon>0$. Then, there exists a finite Blaschke product $B$ such that
		$$\norm*{B-f}_L<\varepsilon.$$
	\end{namedthm*}
	Our first main result is essentially a ``combination'' of the aforementioned theorems of Fisher and Caratheodory, in the sense that we can find a finite Blaschke products that achieves both approximations.
	\begin{theorem}\label{sim_approx_circle}
		Let $K$ be a compact subset of $\T$ of arclength measure $0$ and $L$ be a compact subset of $\D$. Let also $\vp$ be a continuous function on $K$ with $|\vp|\equiv1$ on $K$ and $f\in\mathcal{B}$. Finally, let $\varepsilon>0$. Then, there exists a finite Blaschke product $B$ such that
		$$\norm*{B-\vp}_K<\varepsilon\ \text{ and }\ \norm*{B-f}_L<\varepsilon.$$
	\end{theorem}
	Next, we want to consider the case where $\vp$ is not necessarily unimodular, but is rather contained in the closed unit ball of $C(K)$, which we denote by $BC(K)$. Clearly, the approximation can then not be possible as finite Blaschke products are unimodular on $\T$. But, by dilating the Blaschke products, and thus examining their values inside the unit disc, we get the the following result. For $r\in(0,1)$ and $f\in A(\D)$ we denote by $f_r\in\D$, the dilate of $f$, that is, the function $\overline{\D}\ni z\mapsto f(rz)\in\C$.
	\begin{theorem}\label{sim_approx_disc}
		Let $K$ be a compact subset of $\T$ of arclength measure $0$ and $L$ be a compact subset of $\D$. Let also $\vp\in BC(K)$ and $f\in\mathcal{B}$. Finally, let $\varepsilon>0$. Then, there exists $r_0\in(0,1)$ such that, for any $r\in(r_0,1)$, there exists a finite Blaschke product $B$ such that
		$$\norm*{B_r-\vp}_K<\varepsilon\ \text{ and }\ \norm*{B-f}_L<\varepsilon.$$
	\end{theorem}
	The proofs of the last two theorems are given in the next two subsections.
	\subsection{Proof of Theorem \ref{sim_approx_circle}}
	First, we will prove the special case $f\equiv1$ before proceeding with the result for a general function $f$. For $w\in\D$, we will denote by $\phi_w$ the map $\overline{\D}\ni z\mapsto\frac{z-w}{1-\overline{w}z}\in\overline{\D}$. 
	\begin{prop}\label{simul_approx1}
		Let $K$ be a compact subset of $\T$ of arclength measure $0$ and $L$ be a compact subset of $\D$. Let also $\vp$ be a continuous function on $K$ with $|\vp|\equiv1$ on $K$ and $\varepsilon>0$. Then, there exists a finite Blaschke product $B$ such that
		$$\norm*{B-\vp}_K<\varepsilon\ \text{ and }\ \norm*{B-1}_L<\varepsilon.$$
	\end{prop}
	\begin{proof}
		Let $(w_n)_n$ be a sequence in $(0,1)$ with $\lim w_n=1$. For any $n\in\N$, let $\psi_n\in C(K)$ be the function given by $K\ni z\mapsto\frac{\phi_{w_n}(\vp(z))}{z}\in\T$. By \hyperref[fisher_cont]{Fisher's Theorem}, there exists a finite Blaschke product $B_0$ such that
		$$\norm*{B_0-\psi_n}_K<(1-w_n)^2$$
		and so
		$$\sup_{z\in K}\left|zB_0(z)-\vp_{w_n}(\vp(z))\right|<(1-w_n)^2.$$
		It is easy to check that the function $\phi_{w_n}^{-1}=\phi_{-w_n}$ is Lipschitz continuous on $\overline{\D}$ with Lipschitz constant $\frac{1+w_n}{1-w_n}$. Thus, we get that
		\begin{align*}
			\begin{autobreak}
				\MoveEqLeft[0]
				\sup_{z\in K}\left|\phi_{w_n}^{-1}(zB_0(z))-\vp(z)\right|
				=\sup_{z\in K}\left|\vp_{w_n}^{-1}(zB_0(z))-\phi_{w_n}^{-1}(\phi_{w_n}(\vp(z)))\right|
				\leq\frac{1+w_n}{1-w_n}\sup_{z\in K}\left|zB_0(z)-\phi_{w_n}(\vp(z))\right|<(1+w_n)(1-w_n).
			\end{autobreak}
		\end{align*}
		We note that the function $f_n:\overline{\D}\to\C$ with $f_n(z)=\vp_{w_n}^{-1}(zB_0(z))$ for $z\in\overline{\D}$ is a composition of a disc automorphism with a finite Blaschke product and is thus also a finite Blaschke product (see Theorem 3.6.1 in \cite{GarciaMashreghiRoss2018}). Moreover, since $\lim w_n=1$, there exists $N\in\N$ such that
		$$\norm*{f_n-\vp}_K<\varepsilon\ \text{ for any }n\geq N.$$
		Finally, we see that $f_n(0)=w_n$ for any $n\in\N$ and thus $\lim f_n(0)=1$. This gives us that $f_n$ converges to $1$ uniformly on any compact subset of $\D$. Indeed, assume that there exists $\varepsilon>0,r\in(0,1)$ and a subsequence $(f_{n_j})_j$ such that
		\begin{equation}\label{bound_away}
			\norm*{f_{n_j}-1}_{\overline{D(0,r)}}\geq\varepsilon\ \text{ for any }j\in\N
		\end{equation}
		where $D(0,r)$ denotes the open disc centered at 0 with radius $r$. By Montel's Theorem, there exists a further subsequence $(f_{n_{j_k}})_k$ that converges to some holomorphic function $g\in H(\D)$ uniformly on $\overline{D(0,r)}$. We clearly have that $\norm*{g}_{\overline{D(0,r)}}\leq1$ and $g(0)=1$, thus by the Maximum Modulus Principle, $g\equiv1$ which contradicts \eqref{bound_away} and the result follows.
	\end{proof}
	\begin{remark}
		The previous Proposition as well as its proof bare similarities to another result of Fisher. Indeed, Lemma 2 in \cite{Fisher1969} is a simultaneous approximation result in a different setting, where the approximation on the boundary is in the $L^2$ sense with respect to a given measure.
	\end{remark}
	\begin{proof}[Proof of Theorem \ref{sim_approx_circle}]
		Firstly, by \hyperref[caratheo]{Caratheodory's Theorem}, there exists a finite Blaschke product $B_0$ such that
		$$\norm*{B_0-f}_L<\frac{\varepsilon}{2}.$$
		Then, by Proposition \ref{simul_approx1}, there exists a finite Blaschke product $B_1$ such that
		$$\norm*{B_1-\frac{\vp}{B_0}}_K<\varepsilon\ \text{ and }\ \norm*{B_1-1}_L<\frac{\varepsilon}{2}.$$
		The first inequality clearly implies $\norm*{B_0B_1-\vp}_K<\varepsilon$, while from the second one, we get
		\begin{align*}
			\begin{autobreak}
				\MoveEqLeft[0]
				\norm*{B_0B_1-f}_L
				=\norm*{B_0B_1-B_1f}_L+\norm*{B_1f-f}_L
				<\frac{\varepsilon}{2}+\frac{\varepsilon}{2}=\varepsilon.
			\end{autobreak}
		\end{align*}
		By setting $B=B_0B_1$, the proof is complete.
	\end{proof}
	\subsection{Proof of Theorem \ref{sim_approx_disc}}
	Before proving the theorem, we will need some auxiliary results. The first one is a variation of constructions related to the \hyperref[rcthm]{Rudin-Carleson Theorem} (see for example Lemma 3.1 in \cite{BeneteauIvriiManolakiSeco2020}). For $\alpha\in(0,\frac{\pi}{2})$, we will denote by $C_\alpha$ the cone $\{z\in\C\setminus\{0\}:\arg(z)\in(-\alpha,\alpha)\}$ where by $\arg(z)$ for $z\in\C\setminus\{0\}$ we denote the value of the argument contained in $(-\pi,\pi]$.
	\begin{lemma}\label{RCri}
		Let $K$ be a compact subset of $\T$ of arclength measure 0. Let also $U$ an open subset of $\C$ containing $K$ and $\varepsilon>0$. Then, there exists $h\in A(\D)$ such that:
		\begin{enumerate}[(a)]
			\item $|h|<\varepsilon$ on $\overline{\D}\setminus U$;
			\item $h\equiv1$ on $K$;
			\item $\Re(h)\geq0$ on $\overline{\D}$;
			\item $|\Im(h)|<\varepsilon$ on $\overline{\D}$.
		\end{enumerate}
	\end{lemma}
	\begin{proof}
		Let $V$ be an open set such that $K\subseteq\overline{V}\subseteq U$ and $n\geq2$ be an integer such that
		\begin{equation}\label{sin}
			\sin\left(\frac{\pi}{n}\right)<\varepsilon.
		\end{equation}
		Since $K$ has arclength measure 0, we can construct a function $u_1$ in $C^\infty(\T\setminus K)$ taking values in $[1,\infty)$ such that $\lim u_1(z)=+\infty$ as $z\to w$ for any $w\in K$ and $u_1\equiv1$ on $\T\setminus\overline{V}$. Moreover, we can assume that
		\begin{equation}\label{int_1}
			\frac{1}{2\pi}\int_{\T}u_1(\zeta)|d\zeta|<1+\frac{\mathrm{dist}\left(\overline{\D}\setminus U,\T\cap \overline{V}\right)}{2}.
		\end{equation}
		Let $M>1$ be such that
		\begin{equation}\label{Mcond}
			\frac{2}{\sqrt[n]{M-1}-2}<\varepsilon.
		\end{equation} 
		We can also construct a function $u_2$ in $C^\infty(\T)$ taking values in $[1,M]$ such that $u_2\equiv M$ on $\T\setminus\overline{V}$ and $u_2\equiv1$ on $K$. Moreover, we can assume that
		\begin{equation}\label{int_2}
			\frac{1}{2\pi}\int_{\T}u_2(\zeta)|d\zeta|>M-\frac{\mathrm{dist}\left(\overline{\D}\setminus U,\T\cap \overline{V}\right)}{2}.
		\end{equation}
		We can then extend $u_1,u_2$ harmonically on $\D$ (by a slight abuse of notation, we will use the symbols $u_1,u_2$ also for these extensions) noting that $u_1\in C^\infty(\overline{\D}\setminus K)$ and $u_2\in C^\infty(\overline{\D})$. Let $\tilde{u_1}$ (respectively $\tilde{u_2}$) be the harmonic conjugate of $u_1$ (respectively $u_2$) with $\tilde{u_1}(0)=0$ (respectively $\tilde{u_2}(0)=0$). We note that $\tilde{u_1}\in C^\infty(\overline{\D}\setminus K)$ and $\tilde{u_2}\in C^\infty(\overline{\D})$ since the Hilbert transform preserves local smoothness.
		
		Since the functions $u_1,\tilde{u_1}\in C^\infty(\overline{\D}\setminus K)$ are harmonic conjugates, the function $g_1:\D\to\C$ given by $\D\ni z\mapsto u_1(z)+i\tilde{u_1}(z)$ is holomorphic on $\D$ and extends continuously to $\T\setminus K$. The function $g_1-1$ has non-negative real part on $\D$, so, by the Herglotz Representation Theorem, we have, for any $z\in\D\setminus U$:
		\begin{align*}
			\begin{autobreak}
				\MoveEqLeft[0]
				|g_1(z)-1|=\left|i\tilde{u_1}(0)+\frac{1}{2\pi}\int_{\T}(u_1(z)-1)\frac{\zeta+z}{\zeta-z}|d\zeta|\right|
				\leq\frac{1}{2\pi}\int_{\T}(u_1(z)-1)\left|\frac{\zeta+z}{\zeta-z}\right||d\zeta|
				\leq \frac{1}{\pi\cdot\mathrm{dist}\left(\overline{\D}\setminus U,\T\cap \overline{V}\right)}\int_{\T}(u_1(z)-1)|d\zeta|
			\end{autobreak}
		\end{align*}
		and so, by \eqref{int_1}, we get $|g_1(z)-1|<1$ for any $z\in\D\setminus U$. By continuity, we get
		\begin{equation}\label{ineq_g_1}
			|g_1(z)-1|\leq1
		\end{equation}
		for any $z\in\overline{\D}\setminus U$. Similarly, the function $g_1:\D\to\C$ given by $\D\ni z\mapsto u_1(z)+i\tilde{u_1}(z)$ is holomorphic on $\D$ and extends continuously to $\T$. Moreover, by \eqref{int_2}, we get
		\begin{equation}\label{ineq_g_2}
			|g_2(z)-M|\leq1
		\end{equation}
		for $z\in\overline{\D}\setminus U$.
		
		Now, let $j=1,2$, then $g_j$ takes values in the half-plane $\{z\in\C:\Re(z)\geq1\}$, so there exists a single-valued $n$-th root of $g_j$ that takes values in $C_{\frac{\pi}{2n}}$, which we denote by $h_j$. Finally, let $h:\D\to\C$ be the function given by $h(z)=\frac{h_1(z)}{h_1(z)+h_2(z)}$ for $z\in\D$. We clearly have that $h$ is holomorphic on $\D$ and extends continuously to $\T$ and satisfies (b).
		
		Using the fact that the functions $h_1$ and $h_2$ take values in $C_{\frac{\pi}{2n}}$, we can see geometrically that $h$ takes values in $C_{\frac{\pi}{n}}\cap\D$ (and thus (c) is also satisfied). This implies that $|\Im(h)|<\sin\left(\frac{\pi}{n}\right)$ on $\overline{\D}$ and so (d) follows by \eqref{sin}. Finally, fix $z\in\overline{\D}\setminus U$ from \eqref{ineq_g_1} and \eqref{ineq_g_2}, we get respectively
		$$|h_1(z)|\leq2\ \text{ and }\ |h_2(z)|\geq\sqrt[n]{M-1}$$
		and so
		$$|h(z)|=\left|\frac{h_1(z)}{h_1(z)+h_2(z)}\right|\leq\frac{|h_1(z)|}{|h_2(z)|-|h_1(z)|}\leq\frac{2}{\sqrt[n]{M-1}-2}$$
		and (a) follows by \eqref{Mcond}.
	\end{proof}
	\begin{lemma}\label{approxRC1}
		Let $K$ be a compact subset of $\T$ of arclength measure 0 and $\vp\in BC(K)$. Let also $U$ an open subset of $\C$ containing $K$ and $\varepsilon>0$. Then, there exists $g$ in the closed unit ball of $A(\D)$ such that
		$$\norm*{g-\vp}_K<\varepsilon\ \text{ and }\ \norm*{g-1}_{\overline{\D}\setminus U}<\varepsilon.$$
	\end{lemma}
	\begin{proof}
		We assume that $\vp$ does not vanish on $K$ and then the full result follows by the density of the non-vanishing continuous functions in $C(K)$. The function $\vp$ then has a continuous logarithm, that is, there exists $\vp_0\in C(K)$ such that $\vp=e^{\vp_0}$ on $K$. Since $|\vp|\leq1$ on $K$, we have that $\Re(\vp_0)\leq0$ on $K$. Moreover, there exists $m>0$ such that $|\Im(\vp_0)|\leq m$ on $K$. By the \hyperref[rcthm]{Rudin-Carleson Theorem}, $\vp_0$ extends to a function in $A(\D)$ with $\Re(\vp_0)\leq0$ and $|\Im(\vp_0)|\leq m$ on $\overline{\D}$ (where, by a slight abuse of notation, we use the symbol $\vp_0$ also for the extension).
		
		Let $\delta>0$ to be chosen later, then we apply Lemma \ref{RCri} with $\delta$ in place of $\varepsilon$ to get a function $h\in A(\D)$ such that
		\begin{equation}\label{small}
			|h|<\delta \text{ on } \overline{\D}\setminus U,
		\end{equation}
		\begin{equation}\label{equal1}
			h\equiv1 \text{ on } K,
		\end{equation}
		\begin{equation}\label{rpos}
			\Re(h)\geq0 \text{ on } \overline{\D},
		\end{equation}
		\begin{equation}\label{ismall}
			|\Im(h)|<\delta \text{ on } \overline{\D}.
		\end{equation}
		We consider the function $g_0=h\vp_0\in A(\D)$ and $g_1=e^g\in A(\D)$. By \eqref{equal1}, we have that $g_0=\vp_0$ on $K$ and so 
		\begin{equation}\label{equal_K}
			g_1=\vp \text{ on } K.
		\end{equation}
		Next, we have, for $z\in\D$,
		$$\Re(g_0(z))=\Re(h(z))\Re(\vp_0(z))-\Im(h(z))\Im(\vp_0(z))\leq|\Im(h(z))\Im(\vp_0(z))|<\delta m$$
		where we used \eqref{rpos} and \eqref{ismall}. Thus, we get
		\begin{equation}\label{g_1_bound}
			|g_1|<e^{\delta m}\text{ on }\overline{\D}.
		\end{equation}
		Using \eqref{small}, we get, for $z\in\overline{\D}\setminus U$,
		\begin{equation*}
			|g_0(z)|<\delta\norm*{\vp_0}_{\overline{\D}\setminus U}
		\end{equation*}
		and so
		\begin{align}\label{boundrest}
			\begin{autobreak}
				\MoveEqLeft[0]
				|g_1(z)-1|
				=\left|e^{\Re(g_0(z))+i\Im(g_0(z))}-1\right|
				\leq e^{\Re(g_0(z))}\left|e^{i\Im(g_0(z))}-1\right|+\left|e^{\Re(g_0(z))}-1\right|
				\leq e^{|g_0(z)|}\cdot2\sin\left(\frac{\Im(g_0(z))}{2}\right)+e^{|g_0(z)|}-1
				<e^{\delta\norm*{\vp_0}_{\overline{\D}\setminus U}}\delta\norm*{\vp_0}_{\overline{\D}\setminus U}+e^{\delta\norm*{\vp_0}_{\overline{\D}\setminus U}}-1
			\end{autobreak}
		\end{align}
		Now, consider the function $g:=e^{-\delta m}g_1\in A(\D)$. Clearly, \eqref{g_1_bound} implies that $g$ belongs to the closed unit ball of $A(\D)$. By \eqref{equal_K}, we get
		$$\norm*{g-\vp}_K\leq1-e^{-\delta m}.$$
		Similarly, by \eqref{boundrest}, we get
		$$\norm*{g-1}_{\overline{\D}\setminus U}\leq e^{-\delta m}\left(e^{\delta\norm*{\vp_0}_{\overline{\D}\setminus U}}\delta\norm*{\vp_0}_{\overline{\D}\setminus U}+e^{\delta\norm*{\vp_0}_{\overline{\D}\setminus U}}-1\right)+1-e^{-\delta m}.$$
		Finally, we note that by taking $\delta$ arbitrarily close to 0, the right-hand sides in the last two inequalities also become arbitrarily close to 0 and the proof is complete.
	\end{proof}
	Our next result generalises the previous.
	\begin{prop}\label{approxRC}
		Let $K$ be a compact subset of $\T$ of arclength measure 0 and $\vp\in BC(K)$. Let also $U$ an open subset of $\C$ containing $K$ and $f$ a function in the closed unit ball of $A(\D)$ such that $|f|\equiv1$ on $K$. Finally, let $\varepsilon>0$. Then, there exists $g$ in the closed unit ball of $A(\D)$ such that
		$$\norm*{g-\vp}_K<\varepsilon\ \text{ and }\ \norm*{g-f}_{\overline{\D}\setminus U}<\varepsilon.$$
	\end{prop}
	\begin{proof}[Proof of Theorem \ref{approxRC}]
		By Lemma \ref{approxRC1}, there exists $g$ in the closed unit ball of $A(\D)$ such that
		$$\norm*{g_0-\frac{\vp}{f}}_K<\varepsilon\ \text{ and }\ \norm*{g_0-1}_{\overline{\D}\setminus U}<\varepsilon.$$
		Using the fact that $\norm*{f}_{\overline{\D}\setminus U}\leq1$ and $|f|\equiv1$ on $K$, we easily see that the function $g:=fg_0$ satisfies the desired properties.
	\end{proof}
	Finally, we are ready for the proof of Theorem \ref{sim_approx_disc}.
	\begin{proof}[Proof of Theorem \ref{sim_approx_disc}]
		Taking into account \hyperref[caratheo]{Caratheodory's Theorem}, it is enough to assume that $f$ is a finite Blaschke product, and so it belongs to the closed unit ball of $A(\D)$ and satisfies $|f|\equiv1$ on $K$.
		
		By Proposition \ref{approxRC}, there exists $g$ in the closed unit ball of $A(\D)$ such that
		$$\norm*{g-\vp}_K<\frac{\varepsilon}{3}\ \text{ and }\ \norm*{g-f}_{L}<\frac{\varepsilon}{2}.$$
		The function $g$ is uniformly continuous on $\overline{\D}$, so there exists $0<\delta<1$ such that $|g(z)-g(w)|<\frac{\varepsilon}{3}$ whenever $z,w\in\overline{\D}$ with $|z-w|<\delta$. Thus, by setting $r_0=1-\delta$, we have that
		$$\norm*{g_r-g}_K<\frac{\varepsilon}{3}\text{ for any }r\in(r_0,1).$$
		Let $r\in(r_0,1)$. By applying \hyperref[caratheo]{Caratheodory's Theorem} for the function $g$ on the compact set $L\cup rK$, we can get a finite Blashcke product $B$ such that
		$$\norm*{B_r-g_r}_K<\frac{\varepsilon}{3}\ \text{ and }\ \norm*{B-g}_L<\frac{\varepsilon}{2}.$$
		The result follows by the triangle inequality.
	\end{proof}
	\subsection{Approximation by singular inner functions}
	To finish off this section, we prove a simultaneous approximation result involving singular inner functions that is a combination of analogues of theorems \ref{sim_approx_circle} and \ref{sim_approx_disc}. Let $\mathcal{S}$ denote the set of zero-free elements of $\mathcal{B}$. We recall that singular inner functions are functions $S\in\mathcal{S}$ of the form
	$$S(z)=e^{-\int_{\T}\frac{w+z}{w-z}d\mu(w)},z\in\D$$
	for some singular (with respect to the arclength measure) measure $\mu$ on $\T$.
	\begin{theorem}\label{sim_approx_sing}
		Let $K$ be a compact subset of $\T$ of arclength measure $0$ and $L$ be a compact subset of $\D$. Let also $\vp_1$ and $\vp_2$ be two continuous function on $K$ with $|\vp_1|\equiv1$ on $K$ and $\vp_2\in BC(K)$ and $f\in\mathcal{S}$. Finally, let $\varepsilon>0$. Then,
		\begin{enumerate}[(a)]
			\item there exists a singular inner function $S$ that extends continuously to $K$ such that
			$$\norm*{S-\vp_1}_K<\varepsilon\ \text{ and }\ \norm*{S-f}_L<\varepsilon;$$
			\item there exists $r_0\in(0,1)$ such that, for any $r\in(r_0,1)$, there exists a singular inner function $S$ such that
			$$\norm*{S_r-\vp_2}_K<\varepsilon\ \text{ and }\ \norm*{S-f}_L<\varepsilon.$$
		\end{enumerate}
	\end{theorem}
	\begin{proof}
		We will only present the proof of (a). Part (b) is proved in a similar way with minor modifications. Since $f\in\mathcal{S}$, there exists a continuous logarithm $\log(f)$ on $\D$ of $f$ that takes values in the half-plane $\{z\in\C:\Re(z)\leq0\}$. Then, the function $f_0:=\frac{\log(f)+1}{\log(f)-1}$ is in $\mathcal{B}$ and satisfies $f=e^{\frac{f_0+1}{f_0-1}}$. Similarly, the function $\vp_0:=\frac{\log(\vp_1)+1}{\log(\vp_1)-1}$ is in $BC(K)$ and satisfies $\vp_1=e^{\frac{\vp_0+1}{\vp_0-1}}$ and $|\vp_0|\equiv1$ on $K$. The set $\vp_0(K)\cup f_0(L)$ is clearly compact and omits the point 1. Let $\delta_0>0$ be such that the set $A:=\{z\in\overline{\D}:\mathrm{dist}(z,\vp_0(K)\cup f_0(L))\leq\delta_0\}$ does not contain the point 1. By the uniform continuity of the function $A\ni z\mapsto e^{\frac{z+1}{z-1}}\in\C$, there exists $\delta>0$ such that $\left|e^{\frac{z+1}{z-1}}-e^{\frac{w+1}{w-1}}\right|<\varepsilon$ for any $z,w\in A$ with $|z-w|<\delta$. We can also clearly assume that $\delta<\delta_0$.
		
		By Theorem \ref{sim_approx_circle}, there exists a finite Blaschke product $B$ such that
		$$\norm*{B-\vp_0}_K<\delta\ \text{ and }\ \norm*{B-f_0}_L<\delta.$$
		We have that $B(K\cup L)\subseteq A$ and so we get
		$$\norm*{S-\vp_1}_K<\varepsilon\ \text{ and }\ \norm*{S-f}_L<\varepsilon$$
		where $S:=e^{\frac{B+1}{B-1}}$. We note that $S$ is a singular inner function. Indeed, it is easy to check that $|S|\equiv1$ almost everywhere on $\T$ and $S$ does not vanish on $\D$. Finally, we see that $S$ extends to $K$ since $B$ extends to $K$ and takes values away from 1 there.
	\end{proof}
	\section{Universal radial approximation}\label{univ}
	Using the results of the previous section, we can prove the existence (and even the abundance) of inner functions with universal radial limits. Bounded holomorphic functions with universal boundary behaviour have also been studied through the lens of composition operators. In 1954, Heins proved that there exists a sequence $(\phi_n)_n$ of automorphisms of $\D$ and an (infinite) Blaschke product $B$ such that the sequence $\left(B\circ\phi_n\right)_n$ is dense in $\mathcal{B}$ with respect to the locally uniform topology. In \cite{GorkinMortini2004}, the previous was generalised in the following sense: let $(z_n)_n$ be a sequence in $\D$ with $|z_n|\to1$ and consider the automorphisms of $\D$ given by $\phi_n(z)=\frac{z+z_n}{1+\overline{z_n}z},z\in\D$ for $n\in\N$. Then, there exists a Blaschke product $B$ such that the sequence $\left(B\circ\phi_n\right)_n$ is dense in $\mathcal{B}$ with respect to the locally uniform topology. The literature on this topic includes various generalisations and further results. We refer the interested reader to \cite{GorkinMortini2004_2,Mortini2007,BayartGrivauxMortini2008,GorkinLeonSaavedraMortini2008,BayartGorkinGrivauxMortini2009} and the references therein. 
	
	In a similar direction, there is a lot of work regarding the possible wild boundary behaviour of holomorphic functions on $\D$. In 1933, Kierst and Szpilrajn proved the existence of holomorphic functions $f$ on $\D$ with the property that the set $\{f_r(\zeta):r\in(0,1)\}$ is dense in $\C$ for any $\zeta\in\T$, that is, the image of any radius under $f$ is a dense curve in the space \cite{KierstSzpilrajn1933} (we recall that $f_r$ denotes the dilate $\T\ni\zeta\mapsto f(r\zeta)\in\C$ for $r\in(0,1)$). In \cite{BernalGonzalezCalderonMorenoPradoBassas2004}, the previous result was extended by replacing ``any radius'' with ``any continuous path in $\D$ that converges to a point on $\T$''. Later, Bayart proved the existence of holomorphic functions $f$ on $\D$ with the property that for any measurable function $\vp$ on $\T$, there exists a sequence $(r_n)_n$ in $(0,1)$ converging to 1 such that $f_{r_n}\to\vp$ almost everywhere on $\T$ \cite{Bayart2005}.
	
	In 2020, Charpentier introduced the so-called Abel universal functions. These are functions in $H(\D)$ with the property that the set $\left\{f_r:r\in(0,1)\right\}$ is dense in $C(K)$ for any proper compact subset $K$ of $\T$ \cite{Charpentier2020}. Analogues were given in Banach spaces of holomorphic functions (see \cite{Maronikolakis2022} for Hardy, Bergman and Dirichlet spaces and \cite{CharpentierEspoullierZarouf2025} for the Bloch space). Importantly, the universal approximation needs to be restricted to a fixed compact set with appropriate properties depending on the space. For example, in the Hardy space $H^p$ for $1\leq p<\infty$, if $K$ is a compact subset of $\T$ of arclength measure 0, then there exists a function $f\in H^p$ such that the set $\left\{f_r:r\in(0,1)\right\}$ is dense in $C(K)$. The proof of this result (as well as the corresponding results of other Banach spaces of holomorphic functions) is based on the simultaneous approximation results mentioned in the introduction. Finally, we mention that the universal functions mentioned in the previous results form dense $G_\delta$ subsets of the corresponding space.
	
	Motivated by this literature, we prove in this section the existence of functions in $\mathcal{B}$ and $\mathcal{S}$ with universal radial limits, which are essentially analogues of the aforementioned Abel universal functions. We also prove that there are (infinite) Blaschke products as well as singular inner functions with this wild boundary behaviour
	
	Before stating our main result, we make some observations. Let $\tau$ denote the topology on $\mathcal{B}$ of locally uniform convergence on $\D$ (equivalently, the topology of uniform convergence on all compact subsets of $\D$). We note that the spaces $(\mathcal{B},\tau)$ and $(\mathcal{S},\tau)$ are Polish. Indeed, we can see this fact by constructing using standard methods a metric with respect to which $(\mathcal{B},\tau)$ is complete and then note that $(\mathcal{S},\tau)$ is a $G_\delta$ subset of $(\mathcal{B},\tau)$ and thus it is also Polish (Theorem 3.11 in \cite{Kechris1995}). For our purposes, the important point is that we have Baire's Theorem at our disposal when working with these spaces.
	\begin{definition}
		Let $K$ be a compact subset of $\T$. We will say that a function $f\in\mathcal{B}$ belongs to the class $\UU_K$ if, for any $\vp\in BC(K)$, there exists $(r_n)_n$ in $(0,1)$ converging to 1 such that
		$$\norm*{f_{r_n}-\vp}_K\to0.$$
	\end{definition}
	Using the aforementioned work on universal Blaschke products, we can easily get that there exists a Blaschke product that belongs to $\UU_K$ in the case where $K$ is ``small''. Indeed, we can consider the case where $K$ is a singleton.
	\begin{prop}
		There exists a Blaschke product that belongs to the set $\UU_{\{1\}}$.
	\end{prop}
	\begin{proof}
		Let $z_n=\frac{1}{n+1}$ for $n\in\N$ and consider the automorphisms of $\D$ given by $\phi_n(z)=\frac{z+z_n}{1+\overline{z_n}z},z\in\D$ for $n\in\N$. By Theorem 3.1 of \cite{GorkinMortini2004}, there exists a Blaschke product $B$ such that the sequence $\left(B\circ\phi_n\right)_n$ is dense in $\mathcal{B}$ with respect to the locally uniform topology. In particular, the set $\left\{B\circ\phi_n(0):n\in\N\right\}=\left\{B(z_n):n\in\N\right\}$ is dense in $\overline{\D}$. This clearly implies that $B\in\UU_{\{1\}}$ (indeed $\UU_{\{1\}}$ is exactly the class of functions in $\mathcal{B}$ that map the radius that connects 0 with 1 to a dense subset of $\overline{\D}$).
	\end{proof}
	The boundary behaviour of holomorphic functions and in particular Blaschke products can also be studied through the notion of cluster sets. More precisely, if $f$ is a holomorphic function on $\D$ and $\zeta\in\T$, we define the radial cluster set of $f$ at $\zeta$ to be the set
	$$\left\{w\in\C:w=\lim_{n\to\infty}f(r_n\zeta)\text{ where }(r_n)_n\text{ is a sequence in }(0,1)\text{ such that }r_n\to1\right\}.$$
	We refer the interested reader to the classical book of Collingwood and Lohwater \cite{CollingwoodLohwater1966} on cluster sets. If $B$ is a Blaschke product, then we clearly have that the radial cluster set of $B$ at any point $\zeta\in\T$ is a non-empty, compact and connected subset of $\D$. A theorem of Berna, Colwell and Piranian tells us that, for any sequence $(\zeta_n)_n$ in $\T$ and any sequence $(T_n)_n$ of non-empty, compact and connected subsets of $\overline{\D}$, there exists a Blaschke product $B$ that has the radial cluster set $T_n$ at $\zeta_n$ for all $n\in\N$ \cite{BelnaColwellPiranian}. In particular, choosing $T_n=\overline{\D}$ for any $n\in\N$ gives us the existence of a Blaschke product $B$ that belongs to $\UU_{\zeta_n}$ for all $n\in\N$.
	
	Our main result in this section essentially states that we can generalise this to compact sets $K$ of arclength measure 0, where we also have uniformity with respect to $\zeta\in K$.
	\begin{theorem}\label{univ_main}
		Let $K$ be a compact subset of $\T$ of arclength measure 0. Then, we have the following:
		\begin{enumerate}[(a)]
			\item the set $\ \UU_K$ is a dense $G_\delta$ subset of $(\mathcal{B},\tau)$;
			\item there exists an (infinite) Blaschke product that belongs to $\UU_K$;
			\item the set $\ \UU_K\cap\mathcal{S}$ is a dense $G_\delta$ subset of $(\mathcal{S},\tau)$;
			\item there exists a singular inner function that belongs to $\UU_K$.
		\end{enumerate}
	\end{theorem}
	\begin{proof}
		(a). The set $BC(K)$ is separable (as it is a subset of the separable metric space $C(K)$). Let $(\vp_m)_m$ be a dense sequence in $BC(K)$. For any $m,n\in\N$ and $r\in(0,1)$ we denote by $U_{m,n,r}$ the set
		$$\left\{f\in\mathcal{B}:\norm*{f_r-\vp_m}_K<\frac{1}{n+1}\right\}.$$
		It is easy to check that $\mathcal{U}_K=\cap_{t\in(0,1)}\cap_{m,n\in\N}\cup_{r\in(t,1)}U_{m,n,r}$. Thus, the proof will be complete if we show that, for any $t\in(0,1)$ and $m,n\in\N$, the set $\cup_{r\in(t,1)}U_{m,n,r}$ is open and dense in $\mathcal{B}$. The openness is also easy to see (indeed the sets $U_{m,n,r}$ are open for any $r\in(0,1)$ by the definition of the locally uniform topology). Thus, it remains to show that $\cup_{r\in(t,1)}U_{m,n,r}$ is dense for fixed $t\in(0,1)$ and $m,n\in\N$.
		
		Let $g\in\mathcal{B}$ and $L$ be a compact subset of $\D$. Let also $\varepsilon>0$. Then, by Theorem \ref{sim_approx_disc}, there exists a finite Blaschke product $B$ and $r\in(t,1)$ such that
		$$\norm*{B_r-\vp_m}_K<\frac{1}{n+1}\ \text{ and }\ \norm*{B-g}_L<\varepsilon$$
		and so
		$$B\in\cup_{r\in(t,1)}U_{m,n,r}\ \text{ and }\ \norm*{B-g}_L<\varepsilon$$
		which completes the proof.
		
		(b). We will construct the desired Blaschke product inductively. By Theorem \ref{sim_approx_disc}, we get a finite Blaschke product $B^{(0)}$, and $r_0\in(0,1)$ such that
		\begin{equation}
			\norm*{B^{(0)}_{r_0}-\vp_0}_K<1.
		\end{equation}
		Now, for some $k\in\N$, assume that we have found finite Blaschke products $B^{(0)},B^{(1)},\dots,B^{(k)}$ and $0<r_0<r_1<\dots<r_k<1$. The function $B^{(0)}\cdots B^{(k)}$ is a finite Blaschke product, so it has a finite number of zeros in $\D$. Thus, there exists $R\in(0,1)$ such that the function $\frac{1}{B^{(0)}\cdots B^{(k)}}$ is continuous (and so also uniformly continuous) on $\{z\in\C:R\leq|z|\leq1\}$. Thus, there exists $R_0\in(R,1)$ with the property that 
		\begin{equation}\label{un_cont}
			\norm*{\frac{1}{B^{(0)}_r\cdots B^{(k)}_r}-\frac{1}{B^{(0)}\cdots B^{(k)}}}_K<\frac{1}{2(k+2)}
		\end{equation}
		for any $r\in(R_0,1)$.
		
		By Theorem \ref{sim_approx_disc}, we get a finite Blaschke product $B^{(k+1)}$ and $r_{k+1}\in\left(\max\{r_k,R_0\},1\right)$ such that
		\begin{equation}\label{boundary}
			\norm*{B_{r_{k+1}}^{(k+1)}-\frac{\vp_{k+1}}{B^{(0)}\cdots B^{(k)}}}_K<\frac{1}{2(k+2)}
		\end{equation}
		and
		\begin{equation}\label{conv}
			\norm*{B^{(k+1)}-1}_{\overline{D(0,r_k)}}<\frac{1}{2^{k+1}}.
		\end{equation}
		We note that in the construction we can assume that $r_k\to1$ as $k\to\infty$. Let $B:=\prod_{k=0}^\infty B^{(k)}$. By \eqref{conv}, the infinite product is convergent and so $B$ is an infinite Blaschke product (see Theorem 15.6 in \cite{Rudin1987}). Let $\vp\in BC(K)$. By the density of the sequence $(\vp_m)_m$ in $BC(K)$, there exists a sequence $(\lambda_n)_n$ in $\N$ such that
		\begin{equation}
			\norm*{\vp_{\lambda_n}-\vp}_K\to0\ \text{ as }\ n\to\infty.
		\end{equation}
		Let $n\in\N$ be fixed. We then have
		\begin{equation}
			\norm*{B_{r_{\lambda_n}}-\vp_{\lambda_n}}_K
			=\norm*{\prod_{k=0}^{\lambda_n}B_{r_{\lambda_n}}^{(k)}\prod_{k=\lambda_n+1}^{\infty}B_{r_{\lambda_n}}^{(k)}-\vp_{\lambda_n}}_K
			\leq\norm*{\prod_{k=0}^{\lambda_n}B_{r_{\lambda_n}}^{(k)}-\vp_{\lambda_n}}_K
			+\norm*{\prod_{k=\lambda_n+1}^{\infty}B_{r_{\lambda_n}}^{(k)}-1}_K.
		\end{equation}
		By \eqref{un_cont} and \eqref{boundary}, we get
		\begin{align}
			\norm*{\prod_{k=0}^{\lambda_n}B_{r_{\lambda_n}}^{(k)}-\vp_{\lambda_n}}_K
			&\leq\norm*{B_{r_{\lambda_n}}^{(\lambda_n)}-\frac{\vp_{\lambda_n}}{B^{(0)}_{r_{\lambda_n}}\cdots B^{(\lambda_n-1)}_{r_{\lambda_n}}}}_K\nonumber\\ 
			&\leq\norm*{B_{r_{\lambda_n}}^{(\lambda_n)}-\frac{\vp_{\lambda_n}}{B^{(0)}\cdots B^{(\lambda_n-1)}}}_K+\norm*{\frac{1}{B^{(0)}_{r_{\lambda_n}}\cdots B^{(k)}_{r_{\lambda_n}}}-\frac{1}{B^{(0)}\cdots B^{(k)}}}_K\nonumber\\ 
			&<\frac{1}{\lambda_n+1}.
		\end{align}
		By an inductive argument using \eqref{conv}, we also get
		\begin{equation}
			\norm*{\prod_{k=\lambda_n+1}^{\infty}B_{r_{\lambda_n}}^{(k)}-1}_K<\sum_{k=\lambda_n+1}^{\infty}\frac{1}{2^k}.
		\end{equation}
		By combining the last four numbered expressions, we get that
		\begin{equation*}
			\norm*{B_{r_{\lambda_n}}-\vp_{\lambda_n}}_K\to0\ \text{ as }\ n\to\infty
		\end{equation*}
		and so $B\in\UU_K$.
		
		Parts (c) and (d) can proved in a similar way to (a) and (b) respectively with obvious modifications. The details are left to the reader.
	\end{proof}
	\begin{remark}
		Chalmoukis recently showed that the set of Blaschke products is a dense $G_\delta$ subset of $(\mathcal{B},\tau)$ and similarly the set of singular inner functions is a dense $G_\delta$ subset of $(\mathcal{S},\tau)$ \cite{Chalmoukis2024}. Using these results, we can immediately get (b) from (a) and (d) from (c) in the previous theorem after an application of Baire's Theorem. Moreover, we can get in this way the stronger result that the set of Blaschke products (respectively singular inner functions) that belong to the class $\UU_K$ is a dense $G_\delta$ subset of $(\mathcal{B},\tau)$ (respectively $(\mathcal{S},\tau)$). Nonetheless, we include the constructive proofs of (b) and (d) for the sake of completeness.
	\end{remark}
	We note that in the definition of the class $\UU_K$, the numbers $r\in(0,1)$ are not prescribed in any sense. We could alternatively define the class $\UU_K(\rho)$ as follows:
	\begin{definition}
		Let $K$ be a compact subset of $\T$ and $\rho=(r_n)_n$ a sequence in $(0,1)$ converging to 1. We will say that a function $f\in\mathcal{B}$ belongs to the class $\UU_K(\rho)$ if, for any $\vp\in BC(K)$, there exists a subsequence $(r_{\lambda_n})_n$ of $(r_n)_n$ such that
		$$\norm*{f_{r_{\lambda_n}}-\vp}_K\to0.$$
	\end{definition}
	By simple modifications we can prove a stronger version of Theorem \ref{univ_main} where we replace $\UU_K$ with $\UU_K(\rho)$ for any sequence $\rho=(r_n)_n$ in $(0,1)$ that increases to 1. This generalisation is in accordance with the definitions and results in \cite{Charpentier2020} and \cite{Maronikolakis2022}.
	\bibliographystyle{amsplain}
	\bibliography{refs}
\end{document}